\documentclass [11pt,reqno]{amsart}

\usepackage{amsmath,amsfonts,amssymb,amscd,verbatim,delarray,verbatim,calrsfs}

\theoremstyle{plain}
\newtheorem{theorem}{Theorem}[section]

\newtheorem{prop}[theorem]{Proposition}\newtheorem{Proposition}[theorem]{Proposition}
\newtheorem{Lemma}[theorem]{Lemma}

\theoremstyle{definition}
\newtheorem{defn}[theorem]{Definition}
\newtheorem{Definition}[theorem]{Definition}

\newtheorem{remark}[theorem]{Remark}
\newtheorem{example}[theorem]{Example}\newtheorem{Example}[theorem]{Example}

   \newcommand{\ov}{\overline}
\newcommand{\thmref}[1]{Theorem~\ref{#1}}
\newcommand{\secref}[1]{Section~\ref{#1}}

\newcommand{\lemref}[1]{Lemma~\ref{#1}}

\newcommand{\remarkref}[1]{Remark~\ref{#1}}
\newcommand{\defnref}[1]{Definition~\ref{#1}}
\newcommand{\propref}[1]{Proposition~\ref{#1}}

                                                   \DeclareMathOperator{\Gal}{Gal}

\DeclareMathOperator{\Aut}{Aut} 
\DeclareMathOperator{\Bim}{Bim} 
\DeclareMathOperator{\Hom}{Hom}
\DeclareMathOperator{\End}{End}

\newcommand{\BZ}{\mathbb{Z}}
\newcommand{\BC}{\mathbb{C}}

\newcommand{\BP}{\mathbb{P}}
\newcommand{\BF}{\mathbb{F}}\newcommand{\BQ}{\mathbb{Q}}
\newcommand{\BR}{\mathbb{R}}

\newcommand{\al}{\alpha}

\DeclareMathSymbol{\twoheadrightarrow}  {\mathrel}{AMSa}{"10}

\def\RR{{\mathcal R}}

\begin{document}

\title[Complex Tori] {Simple Complex Tori  of Algebraic Dimension 0}

\author{Tatiana Bandman }\address{Department of Mathematics,
Bar-Ilan University,
Ramat Gan, 5290002, Israel}
\email{bandman@math.biu.ac.il}
\author{Yuri  G. Zarhin}
\address{Pennsylvania State University, Department of Mathematics, University Park, PA 16802, USA}
\email{zarhin@math.psu.edu}
\thanks{The  second named author (Y.Z.) was partially supported by Simons Foundation Collaboration grant   \# 585711. This work was finished in
January - May  2022 during his stay at the Max Planck Institut f\"ur Mathematik (Bonn, Germany), whose hospitality and support are gratefully acknowledged.}
\dedicatory{To the memory of Alexey Nikolaevich Parshin}

\begin{abstract}

Using Galois theory, we construct  explicitly (in all complex dimensions $\ge 2$) an infinite family of simple $g$-dimensional complex tori $T$ that enjoy the following properties.

\begin{itemize}
\item
The Picard number of $T$ is $0;$   in particular, the algebraic dimension of $T$ is $0$.
\item
If $T^{\vee}$ is the dual of $T$ then  $\Hom(T,T^{\vee})=\{0\}$.
\item
The automorphism group $\Aut(T)$ of $T$ is isomorphic to $\{\pm 1\} \times \BZ^{g-1}$.
\item
The endomorphism algebra $\End^{0}(T)$ of $T$ is a purely imaginary number field of degree $2g$.
\end{itemize}
 \end{abstract}

\maketitle

\section {Introduction}\label {intro}

 It is known that a ``very general'' complex torus $T$  of dimension $\dim(T)=g\ge 2$ has the  algebraic dimension  $a(T)=0.$ But the  explicit examples of such tori  with $g>2$ are very scarce.   For $g=2$ one may find  the explicit examples of complex tori with algebraic dimension zero 
 in \cite[Appendix]{EF} and \cite[ Example 7.4] {BL}. (All the tori of complex dimension $1$ have algebraic dimension $1$.) 

  The aim of this paper is  to  provide explicit examples of {\sl simple} complex tori $T$ with $a(T)=0$ in all complex dimensions $g\ge 2.$

   The tori   we construct   have some interesting additional properties and may be viewed as non-algebraic analogues of abelian varieties of CM type,
   see \cite[pp. 12--13 and Th. 4.1 on p. 15]{LangCM}. Similar tori played an important role in  C. Voisin's construction of counterexamples to Kodaira's 
   {\sl algebraic approximation problem}   \cite{Vo04,Vo06}, see also \cite{GS}. (We discuss her results about tori in Remark \ref{Voi} below.)
  We start with the following definitions.
  
  \begin{Definition} \label{simple}  A
  positive-dimensional complex torus $X$ is called {\it simple}  if $\{0\}$ and $X$ are the only complex subtori of $X$  (see, e.g., \cite[ Chapter I, Section 7] {BL}). 
\end{Definition}
  
  \begin{defn}
  \label{specialT}
  A complex torus $T$ of dimension $g \ge 2$ is called {\sl special} if it enjoys the following properties.
  \begin{itemize}
\item[(a)]
$T$ is simple and has algebraic dimension 0. In addition,  its endomorphism algebra $\mathrm{End}^{0}(T)=\mathrm{End}(T)\otimes \BQ$ is  
a purely imaginary number field of degree $2g$.
\item[(b)]
The Picard number $\rho(T)$ of $T$ is $0$.
\item[(c)]
If $T^{\vee}$ is the dual of $T$ then $\Hom(T,T^{\vee})=\{0\}.$
In particular, complex tori $T$ and $T^{\vee}$ are not  isogenous.
\item[(d)]
Let $\Aut(T)$ be the automorphism group of the complex Lie group $T$. Then 
$\Aut(T)$ is isomorphic to $\{1,-1\} \times \BZ^{g-1}$. In particular, $\Aut(T)$ 
 is an infinite commutative group, whose torsion subgroup is a cyclic group of order $2$.
\end{itemize}
    \end{defn}
   
Our main result is the following

\begin{theorem}
\label{mainP}
Let $g \ge 2$ be an integer and $E$ a degree $2g$ number field that enjoys the following properties.

\begin{itemize}
\item[(i)] $E$ is purely imaginary;
\item[(ii)]  $E$ has no proper subfields except $\BQ.$
\end{itemize}

  Choose any isomorphism of $\BR$-algebras
  \begin{equation}
  \label{complexS1}
  \Psi: E_{\BR}:=E\otimes_{\BQ}\BR\to \oplus_{j=1}^g\BC=\BC^g
  \end{equation}
and  a $\BZ$-lattice $\Lambda$ of rank $2g$ in $E\subset E_{\BR}$.  
Isomorphism $\Psi$ provides $E_{\BR}$ with the structure of a $g$-dimensional complex vector space.

Then the complex torus $T=T_{E,\Psi,\Lambda}:=E_{\BR}/\Lambda$ is special and its endomorphism algebra $\mathrm{End}^{0}(T)$ is isomorphic to $E$.

\end{theorem}

We present explicit examples of such fields (see Sections \ref{TrunExp}, \ref{SelmerT}, \ref{MoriZ}) for all $g \ge 2$. 

\begin{remark}
Some authors  call  number fields that enjoy the property (ii) of Theorem \ref{mainP} 
{\sl primitive}.  One may view   Proposition \ref{intermediate} below   as a justification of this  terminology.

\end{remark}

\begin{remark}
Suppose that $g \ge 2$ and a degree $2g$ number field $E$ enjoys the properties (i)-(ii) of Theorem \ref{mainP}.
 Let $\Gamma$ be an integer lattice of rank $2g$ in $E$
and $T_0=T_{E,\Psi,\Gamma}$ the corresponding complex torus of dimension $g$.
If $\Lambda$ is any subgroup of finite index in $\Gamma$ then it is also an integer lattice of rank $2g$ in $E\subset E_{\BR}$. By  Theorem \ref{mainP},
all complex tori $T=T_{E,\Psi,\Lambda}$ are special  and $\mathrm{End}^{0}(T) \cong E$.
On the other hand, the set  of all tori
$T_{E,\Psi,\Lambda}$ is precisely the {\sl isogeny class} of $T_0$ (up to an isomorphism).
Let $\mathcal{X}_g \to B_g$ be a versal family of complex tori of dimension $g$  that was constructed in \cite[Sect. 10]{BL}. (Every complex torus of dimension  $g$
appears as its fiber.)
Its base $B_g$ 
is a {\sl homogeneous} $\mathrm{GL}_{2g}(\BR)$-space. Each isogeny class is a $\mathrm{GL}_{2g}(\BQ)$-{\sl orbit} in $B_g$, which is a dense subset of $B_g$ ,
because  $\mathrm{GL}_{2g}(\BQ)$ is a dense subgroup of $\mathrm{GL}_{2g}(\BR)$. Therefore each isogeny class is dense in the {\sl ``moduli space''} $B_g/\mathrm{GL}_{2g}(\BZ)$
of complex tori of dimension $g$. This implies  that the subset of all $g$-dimensional {\sl special} tori 
is dense in the ``moduli space''. 
\end{remark}

\begin{remark}
\label{Voi}
Let $T=V/\Gamma$ be a complex torus of dimension $g \ge 2$ where $V$ is a $g$-dimensional complex vector space and $\Gamma$ is a discrete lattice of rank $2g$ in $V$.
Let $\phi_T$ be a holomorphic endomorphism of the complex Lie group $T$ 
and $\phi_{\Gamma}$ is the endomorphism of $\Gamma$ induced by $\phi_T$. Let 
$f(x) \in \BZ[x]$ be the  characteristic polynomial of $\phi_{\Gamma}$, which is monic of degree $2g$.  Suppose that 
 the polynomial $f(x)$ is separable, has no real roots and its Galois group $\Gal(f)$ over $\BQ$ is the full symmetric group $\mathbf{S}_{2g}$.
Such a pair $(T, \phi_T)$ is called a {\sl scenic torus} in \cite[Sect. 3, p. 271]{GS}.
 C. Voisin \cite[Sect. 1]{Vo04} proved that  a scenic  $T$ is {\sl not} algebraic and its Picard number is $0$.  It follows from \thmref{mainP}   that 
 $T$  is actually  special.  Indeed, 
let $E$ be the $\BQ$-subalgebra of $\End^0(T)$ generated by $\phi_T$. The conditions on $f(x)$ and $\Gal(f)$ imply that $f(x)$ is irreducible and $E\cong \BQ[x]/f(x)\BQ[x]$ is a  purely imaginary number field of degree $2g$.  The condition on $\Gal(f)$ implies
(thanks to  Example \ref{Alt} below) that $E$ has no proper subfields except $\BQ.$    Thus all conditions of \thmref{mainP}  are met. 

\end{remark}

The proof of  \thmref{mainP} is based on  results of  \cite{OZ}.  Properties  (b), (c), (d)   of \defnref{specialT} are consequences of  the following assertions concerning  
 the {\bf   endomorphism algebra } $$\mathrm{End}^0(T)=\mathrm{End}(T)\otimes \BQ$$  of  $T.$
 Recall \cite{OZ} that  $\mathrm{End}^0(T)$
 is a finite-dimensional (not necessarily semisimple) $\BQ$-algebra. 
 
\begin{prop}
\label{PicardZero}
Let $T$ be a complex torus of dimension $g \ge 2$.  Suppose that
$\mathrm{End}^0(T)$
is a degree $2g$ number field that does not contain a subfield of degree $g$.
Then 
\begin{itemize}
\item[(a)]
$T$ is a simple complex torus of algebraic dimension $0$;
\item[(b)]
The  Picard number $\rho(T)$ of $T$ is $0$, i.e., its N\'eron-Severi group $\mathrm{NS}(T)=\{0\}$.
\end{itemize}
\end{prop}

\begin{prop}
\label{HomZero}
Let $T$ be a complex torus of dimension $g \ge 2$.  Suppose that
$\mathrm{End}^0(T)$
is a degree $2g$ number field that does not contain a proper subfield except $\BQ$.

If $T^{\vee}$ is the dual of $T$ then $\Hom(T,T^{\vee})=\{0\}$.
In particular, $T$ is not isogenous to  $T^{\vee}$.
\end{prop}

\begin{prop}
\label{AutE}
Let $T$ be a complex torus of positive dimension.  Suppose that the endomorphism algebra
$\mathrm{End}^0(T)$
is a purely imaginary number field of degree $2s$ that does not contain roots of unity except $\{1,-1\}$.
Let $\Aut(T)$ be the automorphism group of the complex Lie group $T$.

 Then 
$\Aut(T)$ is isomorphic to
$\{\pm 1\} \times \BZ^{s-1}$. In particular, $\Aut(T)$  is commutative and its torsion subgroup is a cyclic group of order $2$.

\end{prop}

 As a by-product  we get  examples of poor manifolds for any dimension $>1$.

 The notion of a { poor} manifold was
  introduced  in \cite{BZ20}.   It is a complex compact connected manifold containing neither rational curves nor analytic subsets of codimension 1 (and, {\it a fortiori,}  having algebraic dimension 0). It was proven in \cite{BZ20}  that  for a $\BP^1-$bundle $X$   over a poor manifold  $Y$  the  group $\Bim(X)$
 of its bimeromorphic selfmaps coincides with the group
 $\Aut(X)$ of its biholomorphic  automorphisms; the latter
  has the commutative identity  component $\Aut_0(X) $ and the order of any finite subgroup of the quotient
$\Aut(X)/\Aut_0(X) $ is bounded by a constant depending on $X$ only.

    As it was mentioned in \cite{BZ20}, a complex torus   $T$ has algebraic dimension $a(T)=0$  if and only if it is poor.
      There exists an explicit example of a $K3$ surface  that 
    does {\sl not} contain analytic curves  (\cite {McM}) and therefore  is poor.  
We  prove the following 

 \begin{theorem}
   \label{mainCM}
   Let $T$ be a complex  torus of dimension $g \ge 2$.  Suppose that 
$\mathrm{End}^0(T)$ contains 
 a degree $2g$ number field $E$ with the same $1$ such that $E$ does not contain a CM subfield.

Then $T$ has algebraic dimension $0$ and therefore is poor. In addition,  there  exist a simple complex torus $S$
and a  positive integer $r$ such that $T$ is isogenous to the self-product $S^r$ of 
$S$.
   \end{theorem}

\begin{remark}  Let us   note an additional property of special tori. 
The notion of the invariant Brauer group $\mathrm{Br}_T(T)$ of a complex toris $T$ was introduced in \cite{OSVZ} (see also \cite{Cao}).
This group is a finite abelian group of exponent  $1$  or $2$.

We claim that $\mathrm{Br}_T(T)=\{0\}$ if $T$ is {\sl special}. Indeed, $\mathrm{Br}_T(T)$ is isomorphic to a subquotient of $\Hom(T,T^{\vee})$ \cite[Sect. 3.3, displayed formula (13) and Prop. 3.19]{OSVZ}.
Since $\Hom(T,T^{\vee})=\{0\}$ for {\sl special} $T$, the group $\mathrm{Br}_T(T)$ is also $\{0\}$.

\end{remark}

The paper is organized as follows.  In  \secref{poortori} we give  some background.  \secref{mainproof} contains proofs  of  main results.  In  \secref{poortori1} we prove \thmref{mainCM}.
   In  Sections \ref{explicitexamples}, \ref{TrunExp}, \ref{SelmerT}, \ref{MoriZ}, \ref{Infinite} we present  a plenty of  explicit examples  of certain number fields that give rise to special  tori. 
   (Notice that explicit examples of simple complex 2-dimensional tori $T$ with $a(T)=0$  and Picard number $0$ were given in \cite[Appendix]{EF} and \cite[ Example 7.4] {BL} in terms of their period lattices.)

   {\bf Acknowledgements}. We are grateful to the referee, whose thoughtful comments helped to improve the exposition.
   
 \section{ A construction from the Galois  Theory}
 \label{poortori}

As usual, $\BZ,\BQ,\BR,\BC$ stand for the ring of integers and fields of rational, real, and complex numbers respectively.
We write $\bar{\BQ}$ for an algebraic closure of $\BQ$.

 Let us recall the properties of  a purely imaginary  number field $E.$

 We may view it as $E=\BQ(\al), $ where $\al\in  E$ and  there is an irreducible    over  $\BQ$  polynomial   $f(x)\in\BQ[x]$  of degree $2g$  such that $f(\al)=0.$  The property of $E$ to be {\sl purely imaginary} means that $f(x)$ has {\sl no} real roots  in $\BC$.    Let $\al_1,\ov{\al_1}\dots,\al_g,\ov{\al_g}$ be roots of $f(x)$ (here $\ov{\alpha}_j$ stands for the complex conjugate of $\alpha_j$).
 There are $2g$  field embeddings $E\hookrightarrow\BC, $ namely,  two for every $j, 1\le  j\le g:$
 
  $$\sigma_j:1\to 1, \al\to \al_j$$ and $$ \ov{\sigma_j}:1\to 1, \al\to \ov{\al_j}.$$   
For every choice  of $g-$tuple $(\tau_1, \dots, \tau_g), $ where each $\tau_j$ is either $\sigma_j$ or 
$\ov{\sigma_j} $ 
 we  define an injective $\BQ$-algebra homomorphism 
 
 \begin{equation}\label{Psi1000}
\Psi:   E\hookrightarrow \oplus_{j=1}^g\BC=\BC^{g},  E\ni  e\mapsto (\tau_1(e), \dots, \tau_g(e))\in \BC^{g}\end{equation}
that extends by $\BR$-linearity to a homomorphism $\Psi :E_{\BR}\to\BC^g$ of $\BR$-algebras (we keep the notation $\Psi$).
Actually, $\Psi$ is an isomorphism of $\BR$-algebras. Indeed,  let $\{\beta_1, \dots, \beta_{2g}\}$ be a basis of the $2g$-dimensional
$\BQ$-vector space $E$.
It is proven in \cite[Proof of Th. 4.1 on pp. 15--16]{LangCM} that the $2g$-element set 
$$\{\Psi(\beta_1), \dots, \Psi(\beta_{2g})\}\subset \BC^g$$
is linearly independent over $\BR$. It follows
that the image $\Psi(E_{\BR})$ has $\BR$-dimension $2g$. Since
$$\dim_{\BR}(E_{\BR})=2g=\dim_{\BR}(\BC^g),$$
$\Psi :E_{\BR}\to\BC^g$ is an {\sl isomorphism} of  $\BR$-algebras.
 There are precisely $2^{g}$    isomorphisms of $\BR$-algebras $E_{\BR}$ and $\BC^g$
 of the form $\Psi=(\tau_1, \dots, \tau_g),$ where $\tau_i$ are defined in \eqref{Psi1000}.
 We will use these isomorphisms in order to construct complex tori $E_{\BR}/\Gamma$
 with needed properties, where $\Gamma$ is a discrete lattice of maximal rank in $E$.
We   will need the following elementary construction from Galois theory.

Let $n \ge 3$ be an integer and $f(x) \in \BQ[x]$  a degree $n$ {\sl irreducible} polynomial. This means that the $\BQ$-algebra
$$K_f=\BQ[x]/f(x)\BQ[x]$$ is a degree $n$ number field. Let $\RR_f\subset \bar{\BQ}$ be the $n$-element set of roots of $f(x)$. If $\alpha \in \RR_f$
then there is  an isomorphism of $\BQ$-algebras
\begin{equation}
\label{isoF}
\Phi_{\alpha}: K_f=\BQ[x]/f(x)\BQ[x] \cong \BQ(\alpha),  x \mapsto \alpha
\end{equation}
where $\BQ(\alpha)$ is the subfield of $\bar{\BQ}$ generated by $\alpha$.   Clearly, $K_f$ (and hence, $\BQ(\alpha)$)
is {\sl purely imaginary} if and only if $f(x)$ has {\sl no} real roots.

Let $\BQ(\RR_f)\subset \bar{\BQ}$ be the splitting field
of $f(x)$, i.e., the subfield   of $\bar{\BQ}$ generated by $\RR_f$.  Then $\BQ(\RR_f)\subset\bar{\BQ}$  is a finite Galois extension of $\BQ$ containing $\BQ(\alpha)$.
We write $G=\mathrm{Gal}(f)$ for the Galois group $\mathrm{Gal}(\BQ(\RR_f)/\BQ)$ of $\BQ(\RR_f)/\BQ$, which may be viewed as a certain subgroup of the group $\mathrm{Perm}(\RR_f)$ of permutations of $\RR_f$. The irreducibility of $f(x)$ means that $\mathrm{Gal}(f)$ is a 
{\sl transitive} permutation subgroup of   $\mathrm{Perm}(\RR_f)$. Let us consider the stabilizer subgroup
\begin{equation}
\label{Galpha}
G_{\alpha}=\{\sigma \in G\mid \sigma(\alpha)=\alpha\}\subset G.
\end{equation}
Clearly, $G_{\alpha}$ coincides with the Galois group $\mathrm{Gal(\BQ}(\RR_f)/\BQ(\alpha))$ of Galois extension $\BQ(\RR_f)/\BQ(\alpha)$.
If one  varies $\alpha$ in $\RR_f$ then all the subgroups $G_{\alpha}$ constitute a conjugacy class in $G$.

The following assertion is certainly well known but we failed to find a suitable reference.

\begin{prop}
\label{intermediate}
The following conditions are equivalent.
\item[(i)] $K_f$ has no proper subfields except $\BQ.$
\item[(ii)]   $\BQ(\alpha)$ has no proper subfields except $\BQ.$
\item[(iii)]  
$G_{\alpha}$ is a maximal subgroup in $G$.
\item[(iv)]
$G$ is a primitive permutation subgroup of $\mathrm{Perm}(\RR_f)$.
\end{prop}

\begin{remark}
\label{permuatation}
\begin{itemize}
\item[]
\item[(1)]
A transitive permutation group $G$ is {\sl primitive} if and only if the stabilizer of a point is a maximal subgroup of $G$
\cite[Prop. 3.4 on p. 15]{Passman}.
\item[(2)]
Every $2$-transitive permutation group is primitive
\cite[Prop. 3.8 on p. 18]{Passman}.
\end{itemize}
\end{remark}

\begin{proof}[Proof of Proposition \ref{intermediate}]
It follows from \eqref{isoF} that (i) and (ii) are equivalent. It follows from Remark \ref{permuatation}(1) that
(iii) and (iv) are equivalent.

Let us prove that (ii) and (iii) are equivalent. Let $H$ be a subgroup of $G$ that contains $G_{\alpha}$. Let us consider the subfield of $H$-invariants
$$F:=\BQ(\RR_f)^{H}=\{e \in \BQ(\RR_f)\mid \sigma(e)=e \ \forall \sigma \in H\}\subset \BQ(\RR_f).$$
Clearly $F$ is contained $\BQ(\RR_f)^{G_{\alpha}}=\BQ(\alpha)$.

There is the canonical  bijection (Galois correspondence) between the set of  subfields  $\BQ(\RR_f)$ and  the set of the subgroups of $G$  (see e.g., \cite[Chapter VI, Theorem 1.1]{Lang}).
  If $H$ is neither $G_{\alpha}$ nor $G$ (i.e., $G_{\alpha}$ is {\sl not} maximal)  then $F$ is neither $\BQ(\alpha)$
 nor $\BQ(\RR_f)^{G}=\BQ$.  This means if (iii) does not hold then (ii) does not hold as well.
 
 Conversely,  let $F$ be a field that lies strictly between $\BQ(\alpha)$ and $\BQ$. Then the Galois group
 $H:=\mathrm{Gal}(\BQ(\RR_f)/F)$ is a proper subgroup of $G$ that contains $G_{\alpha}$ but does {\sl not} coincide with it.
 Hence $G_{\alpha}$ is {\sl not} maximal. This means that if (ii) does not hold then (iii) does not hold as well.
 This ends the proof.
\end{proof}

\begin{example}
\label{Alt}
Suppose that $n\ge 4$.  Let $\mathrm{Alt}(\RR_f)$ be the only index two subgroup of $\mathrm{Perm}(\RR_f)$, which is 
isomorphic to the {\bf alternating group} $\mathbf{A}_n$.
Then both $\mathrm{Perm}(\RR_f)$ and $\mathrm{Alt}(\RR_f)$ are doubly transitive permutation groups  \cite{Passman} and therefore   are  primitive.
It follows from Proposition  \ref{intermediate} that if $\mathrm{Gal}(f)$  coincides with either $\mathrm{Perm}(\RR_f)$ or $\mathrm{Alt}(\RR_f)$
then $K_f$ does not contain a proper subfield except $\BQ$. In other words, $K_f$ does not contain a proper subfield except $\BQ$
if $\mathrm{Gal}(f)$ is isomorphic either to the full symmetric group $\mathbf{S}_{n}$ or to the alternating group $\mathbf{A}_{n}$. (The case of  $\mathbf{S}_{n}$ was discussed earlier in \cite[Sect. 3, p. 51]{LO}.)
\end{example}

\section{Proofs of main results}
\label{mainproof}

If $X$ is a complex torus then  its endomorphism algebra $\mathrm{End}^0(X)=\mathrm{End}(X)\otimes\BQ$
will be denoted also by 
 $D(X)$ in order to be consistent with the  notation   in \cite{OZ}. 
 
  In the following definition, {\bf smCM} is short for
 {\sl sufficiently many Complex Multiplications}: this terminology is inspired by 
  a similar notion for abelian varieties introduced by F. Oort \cite{Oort}.
 \begin{defn}  Let $T$ be a positive-dimensional complex torus and $E$ a number field of degree $2\dim(T)$. We say that $T$ is a {\bf smCM}-torus or a {\bf smCM}-torus with multiplication by $E$, if there is a $\BQ$-algebra embedding $E \hookrightarrow  D(T)$ that sends $1\in E$ to 
 the identity automorphism of $T$. 
 \end{defn}
 
 The following assertion is contained in \cite[Corollary 1.7 on p. 15]{OZ}.
 \begin{footnote}
{There is a typo in the assertion 2 of this Corollary. Namely, one should read  in the displayed formula 
$[D(S):\BQ]$ (not $[E:\BQ]$).}
\end{footnote}
 
 \begin{Proposition}
 \label{OZcor}
Let  $T$ be  a {\bf smCM}-torus  with multiplication by a number field $E$. 

 Then there are a simple complex torus $S$ and a positive integer $r$ such that 
\begin{enumerate}
\item  $r$ divides $2\dim(T);$\item  $T$ is isogenous to $S^r;$
\item \begin{equation}
\label{degreeOZ}
[D(S):\BQ]=2\dim(S);
\end{equation}
\item the field $E$ contains a subfield, $\mathrm{HDG}(T),$ that is isomorphic to $D(S)$,
and
\begin{equation}
\label{power}
r=\frac{2\dim(T)}{\dim_{\BQ}(\mathrm{HDG}(T))}.
\end{equation}
\end{enumerate}
 \end{Proposition}

  The next lemma is an almost immediate corollary of \propref{OZcor}.

\begin{Lemma}
\label{bigEnd}
Let $T$ be a {\bf smCM}-torus     with  multiplication by a field $E\subset D(T)$.   Suppose that at least one of the following conditions holds.

\begin{itemize}
\item[(i)]  $D(T)=E$.

\item[(ii)] 
$E$ has no proper subfields except $\BQ$.
\end{itemize}
Then  $T$ is a simple torus and $D(T)=E$.

\end{Lemma}

\begin{proof}[Proof of Lemma \ref{bigEnd}]
By
Proposition \ref{OZcor},
there are a simple complex torus $S$ and a positive integer $r$ with properties (1-4) of \propref{OZcor}.

This implies that $D(T)$ is isomorphic to the matrix algebra $\mathrm{Mat}_r(D(S))$ of size $r$ over $D(S)$.
In particular, $D(T)$ is not a field if $r>1$.  This implies readily that in case (i) of Lemma \ref{bigEnd}
$r=1$ and therefore $T$ is isogenous to simple $S$
and therefore is simple itself; by assumption, $D(T)=E$.

Let us do the case (ii).  
The absence of intermediate subfields in $E$ implies that
either $D(S)=\BQ$ or $D(S) \cong E$. In light of \eqref{degreeOZ}, $[D(S):\BQ]$ is {\sl even},
which implies that $D(S)\cong E$ and, therefore,
\begin{equation}
\label{ratio}
\dim_{\BQ}(\mathrm{HDG}(T))=[E:\BQ]=[D(S):\BQ]=2g=2\dim(T).
\end{equation}
It follows that $\dim_{\BQ}(\mathrm{HDG}(T))=2\dim(T)$. Now \eqref{power} implies that $r=1$,
hence, $T$ is isogenous to    simple $S$ and, therefore, is a simple torus itself. In addition,
\begin{equation}
\label{endoField}
D(T) \cong D(S)\cong E.
\end{equation}
So, the $\BQ$-algebra $D(T)$ is isomorphic to its subfield $E$ and therefore coincides with $E$.\end{proof}

  \begin{Lemma}
   \label{simpleNotCM}
   Let $T$ be a simple complex torus  of positive dimension $g$ such that
  its endomorphism algebra  $D(T)$ is a degree $2g$ number field $E$ that is not CM.
      Then  $a(T)=0.$ 
      \end{Lemma}

   \begin{proof}
  Every complex torus $T$ admits a maximal quotient abelian variety $T_a$ such that  $\dim T_a=a(T)$  (\cite[Ch. 2, Sect. 6]{BL}).
 The  (connected) kernel of the surjective  homomorphism $T\to  T_a$ is a (complex) subtorus of $T. $ Thus, if $T$ is simple,  either it is an abelian variety or $a(T)=0$. 
  Suppose $T$ is an abelian variety. Then Albert's classification of endomorphism algebras of simple complex abelian varieties \cite[Section 21, Application I]{Mum} implies that
  $E=D(T)$ has degree $[E:\BQ] \le 2g$; if the equality holds then $E$ is a CM field.
  Since $E$ has degree $2g$ but is not a CM field, we get a contradiction that proves that $a(T)=0$.
   \end{proof}
 
\begin{remark}
\label{nonCMii}
If   $T$ is a  smCM-torus with   multiplication by a field $E$,   $g=\dim(T)>1, $  and  condition (ii) of \lemref{bigEnd}  holds then  $E$ is not a CM field, 
because  a degree $2g$ CM field contains a (totally) real subfield of degree $g$. 
\end{remark}

\begin{proof}[Proof of Proposition \ref{PicardZero}]
\label{invE}
 We are given that
$E=D(T)$
is a number field of degree $2g, $   hence  $T$ is a smCM torus and condition (i) of \lemref{bigEnd}  holds.   Thus    $T$ is a {\sl simple} complex torus of positive dimension $g$.  
The absence of degree $g$ subfields   in $E$  implies that $E$ is {\sl not} a CM field (see Remark \ref{nonCMii}).
It follows from \lemref{simpleNotCM}  that $T$ has algebraic dimension $0$.

Suppose that $\mathrm{NS}(T) \ne \{0\}$. Then there exists a holomorphic line bundle $\mathcal{L}$ on $T$,
whose first Chern class $c_1(\mathcal{L})\ne 0$. Then $\mathcal{L}$ gives rise to a {\sl nonzero} morphism of complex tori
$$\phi_{\mathcal{L}}: T \to T^{\vee}$$
where the $g$-dimensional complex torus $T^{\vee}=\mathrm{Pic}^0(T)$ is the dual of $T$ (see \cite[Ch. 2, Sect. 3]{BL}).

 Since $T$ is simple and both $T$ and  $T^{\vee}$ have the same dimension $g$, the  {\sl nonzero} morphism $\phi_{\mathcal{L}}$ is an {\sl isogeny} of complex tori. This means that $T$ is a {\sl nondegenerate}
complex torus  \cite[Ch. 2, Prop. 3.1]{BL} in the terminology of \cite{BL}. Since $T$ is simple,   $\mathcal{L}$ is a ``polarization'' on $T$ (see \cite[Proposition 1.7, Ch. 2, Sect. 1]{BL}).

Let
$$\mathrm{End}^0(T) \to \mathrm{End}^0(T), \ u \mapsto u^{\prime}$$
be the {\sl Rosati involution} attached to $\mathcal{L}$ \cite[Ch. 2, Sect. 3]{BL}. If it is nontrivial then the subalgebra  of its invariants
is a degree $g$ subfield of the field $E=\mathrm{End}^0(T) $  (see \cite[Theorem 1.8, Chapter VI]{Lang}). 
However, by our assumption, such a subfield does not exist.
This implies that the Rosati involution is the identity map.  It follows from \cite[Ch. 5, Prop. 1.2, last assertion]{BL}
that $2g=[E:\BQ]$ divides $g$, which is nonsense.  The obtained contradiction implies that $c_1(\mathcal{L})$
is always $0$, i.e., $\mathrm{NS}(T) = \{0\}$.
\end{proof}

\begin{proof}[Proof of Proposition \ref{HomZero}]
Let us present complex torus $T$ as the quotient
$$T=V/\Gamma$$
where $V$ is a $g$-dimensional complex vector space and $\Gamma$ a discrete additive subgroup of rank $2g$.
Let 
$$\Gamma_{\BQ}:=\Gamma\otimes\BQ,  \ \Gamma_{\BR}:=\Gamma\otimes\BR$$
 be $2g$-dimensional $\BQ$- and $\BR$-vector spaces,  respectively.

 Note that 
 $V\cong \BC^g$ coincides with   $\Gamma_{\BR} $ endowed with complex structure. Namely, there is 
$$J \in \End_{\BR}(\Gamma_{\BR}),$$ which is multiplication by $\mathbf{i}=\sqrt{-1}$ in the $\BC$-vector space $V.$

 Moreover, 
 $$\End_{\BR}(V)=\End_{\BR}(\Gamma_{\BR}), \ \End_{\BC}(V)=\{u\in \End_{\BR}(\Gamma_{\BR}) \ \mid uJ=Ju\}.$$

We have
\begin{equation}
\label{JJ}
J^2=-1, \ J^{-1}=-J.
\end{equation}

It is known  (\cite[Proposition 5.2.11]{Harder})  
that $\End(T)\subset \End(\Gamma)$ and 
\begin{equation}
\label{endT}
 \End(T)\otimes_{\BZ}\BR=\End^0(T)\otimes_{\BQ}\BR=\{u\in  \End_{\BR}(\Gamma_{\BR})\mid uJ=Ju\}.
\end{equation}

In particular,  the $2g$-dimensional $\BQ$-vector space  $\Gamma_{\BQ}$  carries the natural structure of a faithful $\End^0(T)$-module.
Recall that $E=\End^0(T)$ is a number field of degree $2g$. Hence, $\Gamma_{\BQ}$ becomes the one-dimensional $E$-vector space
and therefore $E$ coincides with its own centralizer in $\End_{\BQ}(\Gamma_{\BQ})$.   This implies that if we put
$$E_{\BR}=E\otimes_{\BQ}\BR\subset \End_{\BQ}(\Gamma_{\BQ})\otimes_{\BQ}\BR=\End_{\BR}(\Gamma_{\BR})$$
then $\Gamma_{\BR}$ becomes the free $E_{\BR}$-module of rank $1$ and therefore
$E_{\BR}$ coincides with its own centralizer in $\End_{\BR}(\Gamma_{\BR})$. This implies that
$J \in E_{\BR}.$

\begin{Lemma}
\label{bilinear}
Let $B: \Gamma\times \Gamma \to \BZ$ be a $\BZ$-bilinear form. Let us extend it by $\BR$-linearity to
the $\BR$-bilinear form
$$\Gamma_{\BR} \times \Gamma_{\BR}  \to \BR,$$
which we continue to denote by $B$. Suppose that
\begin{equation}
\label{Jinv}
B(Jv_1,Jv_2)=B(v_1,v_2) \ \forall v_1,v_2\in \Gamma_{\BR} .
\end{equation}
Then $B \equiv 0$.
\end{Lemma}

\begin{proof}[Proof of Lemma \ref{bilinear}]
Clearly, 
\begin{equation}
\label{BQ}
B(\Gamma_{\BQ},\Gamma_{\BQ})\subset \BQ.
\end{equation}
The $J$-invariance of $B$ means that
\begin{equation}
\label{Jminus}
B(Jv_1,v_2)=B(v_1, J^{-1}v_2)=-B(v_1,Jv_2) \   \forall v_1,v_2\in \Gamma_{\BR}
\end{equation}
because $J^{-1}=-J$ (since $J^2=-1$). It follows that the $\BR$-vector subspace
$$E_{\BR}^{-}=\{u\in  E_{\BR} \mid  B(u(v_1),v_2)=-B(v_1,u(v_2)) \   \forall v_1,v_2\in \Gamma_{\BR}\}$$
of $E_{\BR}$ is {\sl not} zero. In light of \eqref{BQ}, there is a nonzero  $\BQ$-vector subspace $E^{-}$ of $E$ such that
$$E_{\BR}^{-}=E^{-}\otimes_{\BQ}\BR.$$
Clearly, $E^{-}=E_{\BR}^{-}\bigcap E$ and
$$E^{-}=\{u\in  E\mid  B(u(v_1),v_2)=-B(v_1,u(v_2)) \   \forall v_1,v_2\in \Gamma_{\BQ}.\}$$
Let $u_{-}$ be a nonzero element of $E^{-}$. Clearly, $$u_{-}\not\in\BQ\subset E.$$ On the other hand,
$$u_{+}:=u_{-}^2\in E$$ also does {\sl not} lie in $\BQ$, because otherwise $\BQ+\BQ\cdot u_{-}$ is a quadratic subfield of $E$, which does {\sl not} contain quadratic subfields.
(Recall that $[E:\BQ]=2g>2$.)  Notice that
$$B(u_{+}(v_1),v_2)=B(v_1,u_{+}(v_2)) \   \forall v_1,v_2\in \Gamma_{\BQ}.$$
Let us consider
$$E^{+}=\{u\in  E_{\BR} \mid  B(u(v_1),v_2)=B(v_1,u(v_2)) \   \forall v_1,v_2\in \Gamma_{\BR}\}.$$
Clearly, $E^{+}$ is a subfield of $E$ that contains $u_{+}$ and therefore does {\sl not} coincide with $\BQ$.
This implies that $E^{+}=E$. It follows that for all $u\in E_{\BR}$
$$B(u(v_1),v_2)=B(v_1,u(v_2)) \   \forall v_1,v_2\in \Gamma_{\BR}.$$
Since $J \in  E_{\BR}$, it follows from \eqref{Jminus} that
$$B(Jv_1,v_2)=0 \   \forall v_1,v_2\in \Gamma_{\BR}.$$
Since $J$ is an automorphism of $\Gamma_{\BR}$, we get
$B \equiv 0.$
\end{proof}

We continue to prove Proposition \ref{HomZero}. Let us recall a description of the dual complex torus $T^{\vee}$ of $T$ (\cite[Ch. 1, Sect. 4]{BL},
\cite[Sect. 1.4]{Kempf}). Namely,  $T^{\vee}=V^{\vee}/\Gamma^{\vee}$ where $V^{\vee}$ is the complex vector space of all $\BC$-antilinear maps
$l: V \to \BC$ and
$$\Gamma^{\vee}=\{l \in V^{\vee}\mid \mathrm{Im}\left(l(\Gamma)\right) \subset \BZ\}.$$
The structure of a complex vector space on $V^{\vee}$ is defined by the operator 
$J^{\vee}\in \End_{\BR}(V^{\vee})$  such that $J^{\vee}(l)=il, $ i.e. $J^{\vee}(l)(v)=il(v).$
  By construction, 

\begin{equation}
\label{anti}
J^{\vee}(l)=-l\circ J \ \forall l \in V^{\vee}
\end{equation}
(recall that $l$ is {\sl antilinear}).

Let $f: T \to T^{\vee}$ be a morphism of complex tori (viewed as complex Lie group). Then (see \cite[Ch. 1, Sect. 1, p. 4]{BL} and \cite[Sect. 3.3]{OSVZ}) there exists  (a lifting of $f$, i.e.,) a $\BC$-linear map
$F: V \to V^{\vee}$ such that $F(\Gamma)\subset \Gamma^{\vee}$ and
$$f(v+\Gamma)=F(v)+\Gamma^{\vee}\in V^{\vee}/\Gamma^{\vee}=T^{\vee}   \ \forall v+\Gamma\in V/\Gamma=T.$$
Let us consider the sesquilinear form
$$H: V \times V \to \BC,  v_1, v_2 \mapsto F(v_1)(v_2)$$
and its imaginary part (which is a $\BR$-bilinear form)
$$B=\mathrm{Im}(H):  V \times V \to \BR, \ v_1, v_2 \mapsto \mathrm{Im}\left((F(v_1)(v_2)\right).$$
Clearly,
$$B(\Gamma,\Gamma)\subset \BZ, \  H(Jv_1,Jv_2)=H(v_1,v_2) \ \forall v_1,v_2 \in V=\Gamma_{\BR}.$$
This implies that 
 $$B(Jv_1,Jv_2)=B(v_1,v_2) \ \forall v_1,v_2 \in V=\Gamma_{\BR}.$$
 By Lemma \ref{bilinear}, $B \equiv 0$.  
 This implies that $H\equiv 0$ (see \bibitem[Lemma 2.1.7]{CAV})
and therefore $F\equiv0$. It follows that $f=0$, which ends the proof.
\end{proof}

\begin{proof}[Proof of Proposition \ref{AutE}]
Clearly,  $\End(T)$ is an order in the purely imaginary number field   $E=\End(T)\otimes\BQ$
of degree $2s$; its group of invertible elements (units) $\End(T)^{*}$ coincides with $\Aut(T).$ It is also clear that the roots of unity in $\End(T)$ are precisely  $1$ and $-1$. Now the desired result 
follows from Dirichlet's theorem about units \cite[Ch. II, Sect. 4, Th. 5]{BS}.
\end{proof}

\begin{proof}[Proof of Theorem \ref{mainP}]
We keep the notation of \thmref{mainP}.
Let us put 
$T:=T_{E,\Psi,\Lambda}$
and consider
\begin{equation}
O:=\{u \in E\mid u\cdot \Lambda\subset \Lambda\}\subset E.
\end{equation}
Then $O$ is an {\sl order} in $E$ \cite[Ch. VII, Sect. 2, Th. 3]{BS}.
Multiplications by elements of $O$ in $E_{\BR}$ give rise to the ring embedding
\begin{equation}
\label{Zembed}
O \hookrightarrow \mathrm{End}(T), 
\end{equation}
which extends by $\BQ$-linearity to the $\BQ$-algebra embedding
\begin{equation}
\label{Qembed}
 E=O\otimes\BQ \hookrightarrow \mathrm{End}(T)\otimes\BQ = D(T). 
\end{equation}
This  allows us to view $E$ as a certain $\BQ$-subalgebra of $D(T).$ Note  that $ 1\in E$ is mapped to $1 \in D(T).$
Recall that
$$[E:\BQ]=2g=2\dim(T).$$
Appying Lemma \ref{bigEnd}, we conclude that $T$ is simple and $D(T)=E$.

Recall that $\dim(T)\ge 2$. Applying \lemref{simpleNotCM} to $T$ and taking into account  \remarkref{nonCMii},
we obtain that the algebraic dimension of $T$ is $0$. It follows from already proven Propositions \ref{PicardZero}  and  \ref{HomZero} that
$\mathrm{NS}(T)=\{0\}$ and $\Hom(T,T^{\vee})=\{0\}$.

In order to prove assertion (d), notice that $E$ does not contain any roots of unity except $\{1,-1\}$. Indeed, if this is not the case
then either $E$ contains  $\sqrt{-1}$ or a primitive $p$th root of unity $\zeta$ where $p$ is an odd prime.
In all these cases $E$ contains a quadratic subfield  that is either $\BQ(\sqrt{-1})$ or $\BQ(\sqrt{-p})$ (if $p$ is congruent to $3\bmod 4$)
or $\BQ(\sqrt{p})$ (if $p$ is congruent to $1\bmod 4$).  Since $E$ does not contain a quadratic subfield, it does not contain 
any roots of unity except $\{1,-1\}$. Now the assertion (d) follows readily from Proposition \ref{AutE}.

\end{proof}

\begin{Example}
\label{basicE}
Let $T=V/\Gamma$ be a complex torus of dimension $g \ge 2$ where $V$ is a $g$-dimensional complex vector space and $\Gamma$ is a discrete lattice of rank $2g$ in $V$.
Let $\phi_T$ be a holomorphic endomorphism of the complex Lie group $T$ that enjoys the following property. 

Let $\phi_{\Gamma}$ is the endomorphism of $\Gamma$ induced by $\phi_T$,
and $f(x) \in \BZ[x]$ the  characteristic polynomial of $\phi_{\Gamma}$ (which is monic of degree $2g$).  Then $f(x)$ is separable, has no real roots and its Galois group $\Gal(f)$ over $\BQ$ is 
a {\sl transitive primitive} subgroup of
the full symmetric group $\mathbf{S}_{2g}$.

Let $E$ be the $\BQ$-subalgebra of $\End^0(T)$ generated by $\phi_T$. The conditions on $f(x)$ and $\Gal(f)$ imply that $f(x)$ is irreducible and $E\cong \BQ[x]/f(x)\BQ[x]$ is a  purely imaginary number field of degree $2g$.  In light of Proposition \ref{intermediate}, the condition on $\Gal(f)$ implies
 that $E$ has no proper subfields except $\BQ.$  Applying Theorem \ref{mainP}, we conclude that $T$ is a {\sl special torus}  and
$\End^0(T)=E$.
\end{Example}

 \section{Poor tori }
 \label{poortori1}

 \begin{Definition}[See \cite{BZ20}] 
 \label{poor}
 We say that a compact connected complex        manifold $Y$ of positive dimension is {\sl  poor} if it enjoys the following properties.
\begin{itemize}
\item
The algebraic dimension $a(Y)$ of $Y$ is $0$.
\item
$Y$ does not contain analytic subspaces of codimension 1.
\item
$Y$   contains  no  rational curve, i.e.,   the   image of a non-constant holomorphic
  map $\BP^1\to Y.$  (In other words, every holomorphic map $\BP^1 \to Y$ is constant.)
\end{itemize}
\end{Definition}

  Let   $Y$ be a poor manifold.   Obviously, $\dim(Y)\ge 2.$  For a surface, {\it poor} means the absence of  any curve $  C\subset Y.$   Explicit examples of $K3$ surfaces having this property may be found in \cite{ McM}   and in \cite[Proposition  3.6, Chapter VIII]{BHPV}).
Explicit examples of complex 2-dimensional tori $Y$ with $a(Y)=0$  are given in \cite[ Example 7.4] {BL}. 
 It  is proven in   \cite[Theorem 1.2]{CDV}  that if a compact   K\"{a}hler  3-dimensional manifold
has no closed subvarieties of dimension 1 or 2 then it is a complex torus. 
 
  On the other hand, a  complex torus  $T$ with  $\dim(T)\ge 2 $ and  $a(T)=0$ is a { poor}  K\"{a}hler manifold.  Indeed, a complex torus  $T$  is a   K\"{a}hler   manifold that does not contain rational curves.
 If $a(T)=0, $ then $T$ contains no analytic subsets of codimension 1 \cite[Corollary 6.4,   Chapter 2]{BL}.   Thus a  complex torus  $T$ is poor if and only if $a(T)=0.$ 
 
 We will use the  following properties of poor manifolds.

   \begin{Lemma}\label{poorprop}
 Let $X,Y$ be two complex compact connected manifolds and let $f:X\to Y$  be a surjective holomorphic map.  Assume that $Y$ is poor.  Then \begin{enumerate}\item if $ F_y:=f^{-1}(y)$  is finite for every $y\in Y$  then $X$ is poor;\item if $F_y:=f^{-1}(y)$ is 
 a poor manifold with $\dim(F_y)=\dim(X)-\dim(Y)$  for every $y\in Y$,  then $X$ is poor.\end{enumerate} In particular, the direct product of poor manifolds is a poor manifold.\end{Lemma}
 \begin{proof}  For proving (1), let us note that $f$ is  an unramified cover of $Y.$  Indeed,  the image $R$  under $f$ of the ramification locus  is either empty or  has pure codimension 1  in $Y$ (\cite[Section 1.1]{DG}, \cite[Theorem1.6]{Pe}, \cite{Re}). Since $Y$ is poor, $R$ is empty.  Now statement (1)  follows from   \cite[Lemma 3.1]{BZ20}. 
 
 Let us prove (2). Assume that $C\subset X$ is a rational curve. If  $C\subset F_y$ for some $y\in Y$ then  $ F_y$  is not poor, which is not the case.   Thus $f(C)$ is a rational curve in $Y$ which is also impossible, since $Y$ is poor. Assume that $D\subset Y$
 is  an analytic irreducible subspace of codimension 1  and $y\in Y.$ If $D\cap F_y\ne\emptyset,$ then $D\supset  F_y, $ since otherwise    $D\cap F_y$ would have codimension one    in $F_y.$    Thus $f(D)$ is an analytic subspace of $Y$  (\cite{Re}, \cite [Theorem 2, Chapter  VII]{Nar}) and $$\dim (f(D))=\dim (D)-\dim (F_y)=\dim(Y)-1,$$  which is impossible since $Y$ is poor. 
 Thus, $X$ is poor: it contains neither rational curves nor analytic  subspaces of codimension 1.\end{proof}
 
\begin{remark}  Let $X$ be as  in Lemma \ref{poorprop}. The fact that  $a(X)=0$ follows also from  \cite[Theorem 3.8]{Ueno}).\end{remark}


\begin{proof}[Proof of Theorem \ref{mainCM}] 
Notice that $T$ is a  smCM-torus  with  multiplication by $E$. Thanks to Proposition \ref{OZcor},
 $T$  is isogenous to $S^r$ where $S$ is a simple torus
such that its endomorphism algebra $D(S)$ is isomorphic to a subfield of $E$.  Hence, $D(S)$ is not a CM field.
Applying Lemma \ref{simpleNotCM}, we conclude that the algebraic dimension of $S$ is $0$.  According to \lemref{poorprop}, this implies that $a(T)=0$ as well.

\end{proof}

\section{Explicit examples}\label{explicitexamples}
Our goal is to describe an explicit construction of special complex tori $T$
in all complex dimensions $g \ge 2$.  
 In order to apply \thmref{mainP} and \propref{intermediate},  let us find a degree $2g$  irreducible polynomial   $f(x)\in\BQ[x]$
such that \begin{itemize}\item $f(x)$ has no real roots;\item $\mathrm{Gal}(f)$ is {\sl primitive}. \end{itemize}

Suppose that we are given such a $f(x)$ (see this section below).  Then the quotient
$E=\BQ[x]/f(x)\BQ[x]$ is a degree $2g$ purely imaginary field that does not contain proper subfields except $\BQ$.
We write $\tilde{x}$ for the image of $x$ in $E$. Then
$\{1, \tilde{x}, \tilde{x}^2, \dots, \tilde{x}^{2g-1}\}$ is a basis of  the $\BQ$-vector space $E$ of dimension $2g$. It follows that
$$\Lambda:=\BZ \cdot 1+ \BZ \cdot \tilde{x} +\dots + \BZ \cdot  \tilde{x}^{2g-1}\subset E$$
is  a free $\BZ$-module of rank $2g$  with the basis
$$\{1, \tilde{x}, \tilde{x}^2, \dots, \tilde{x}^{2g-1}\}.$$

Let $\alpha_1, \dots, \alpha_g \in \BC$ be all the roots of $f(x)$ with positive imaginary part. Then
$$\{\alpha_1, \bar{\alpha}_1, \dots , \alpha_g, \bar{\alpha}_g\}$$
is the set of all complex roots of $f(x)$.  Let
$$\tau_j: E=\BQ[x]/f(x)\BQ[x] \hookrightarrow \BC,  \ u(x)+f(x)\BQ[x] \mapsto u(\alpha_j)$$
be the $\BQ$-algebra homomorphism that sends  $\tilde{x} \in E$ to $\alpha_j$  ($1 \le j \le g$).
As in the beginning of Section \ref{poortori}, the direct sum of all $\tau_j$ defines an injective $\BQ$-algebra homomorphism
$$\Phi: E \hookrightarrow \BC^g,   \beta \mapsto (\tau_1(\beta), \dots \tau_g(\beta)),$$ which extends to the isomorphism 
$\Phi: E_{\BR} \cong \BC^g$ of $\BR$-algebras.  If $\beta \in E \subset E_{\BR}$ then
$$\Phi(\beta)=(\tau_1(\beta), \dots \tau_g(\beta)) \in \BC^g.$$
In particular,
$$\Phi(1)=(1, \dots, 1), \Phi(\tilde{x})=(\alpha_1, \dots, \alpha_g)$$
and therefore 
$$\Phi(\tilde{x}^k)=(\alpha_1^k, \dots, \alpha_g^k)$$
for all nonnnegative integers $k$. This implies that the $2g$-element set 
$$(1, \dots, 1), (\alpha_1, \dots, \alpha_g), (\alpha_1^2, \dots, \alpha_g^2), \dots, (\alpha_1^{2g-1}, \dots, \alpha_g^{2g-1})$$
is a basis of the lattice $\Phi(\Lambda) \subset \BC^g$.   

Let us consider the $g$-dimensional complex torus
$$T(f):=\BC^g/\Phi(\Lambda).$$
It follows from Theorem \ref{mainP} combined with  \propref{intermediate} that  $T(f)$ is a {\sl  special torus} and 
$\mathrm{End}^{0}(T(f))\cong \BQ[x]/f(x)\BQ[x]$.

Below we present such  polynomials for every even degree $2g \ge 2$. (See also an explicit example for $g=4$ in \cite[Sect. 3A, pp. 271--272]{GS}.)

\section{Truncated exponents}
 \label{TrunExp}
Let $n \ge 1$ be an integer. Let us consider the {\sl truncated exponent}
 $$\exp_n(x)=\sum_{j=0}^n \frac{x^j}{j!}\in \BQ[x]\subset \BR[x].$$
 Notice that its derivative
 $$\exp_{n}^{\prime}(x)=\exp_{n-1}(x)=\exp_n(x) - \frac{x^n}{n!} \ \ \forall n \ge 2.$$
 
\begin{Lemma}
\label{noreal}
If $n \ge 2$ is an even integer then $\exp_n(x)$ has no real roots.
\end{Lemma}

\begin{proof}
Since $\exp_n(x)$ is an even degree polynomial with positive leading coefficient, it takes on the smallest possible value on $\BR$ at a certain $x_0\in \BR$. Then 
$$0=\exp_{n}^{\prime}(x_0)=\exp_n(x_0) - \frac{x_0^n}{n!}$$
and therefore
\begin{equation}
\label{x0min}
\exp_n(x_0) = \frac{x_0^n}{n!}.
\end{equation}
If $x_0=0$ then $1=\exp_n(0)=0$, which is not the case.
This implies that $x_0 \ne 0$. Taking into account that $n$ is even, we obtain from \eqref{x0min} that
$\exp_n(x_0)>0$.  Since $\exp_n(x_0)$ is the smallest value of the function $\exp_n$ on the whole $\BR$, the polynomial  $\exp_n$  takes on only positive values on $\BR$ and therefore has {\sl no} real roots.
\end{proof}

By a theorem of Schur \cite{Coleman},  $\mathrm{Gal}(\exp_n(x))=\mathbf{S}_n$ or $\mathbf{A}_n$. 
It follows from Example \ref{Alt} combined with Lemma \ref{noreal}
that if $n=2g$ is even then
$$E=K_g:=\BQ[x]/\exp_{2g}(x)\BQ[x]$$
is a degree $2g$ purely imaginary field that has no proper subfields except $\BQ$.

Now the construction of Section \ref{explicitexamples} applied to $f(x)=\exp_{2g}(x)$ 
gives us for all $g \ge 2$ 
a 
{\sl special} $g$-dimensional complex torus $T(\exp_{2g})$
with endomorphism algebra $K_g=\BQ[x]/\exp_{2g}(x)\BQ[x]$.

\section{ Selmer polynomials}
\label{SelmerT}
 Another series of examples is provided by polynomials
\begin{equation}
\label{selme}
\mathrm{selm}_{2g}(x)=x^{2g}+x+1\in \BZ[x]\subset \BQ[x]\subset \BR[x].
\end{equation}
Notice that $\mathrm{selm}_{2g}(x)$ takes on only {\sl positive values} on the real line $\BR$, hence,
it does {\sl not} have real roots.  Indeed, if $a\in \BR, |a| \ge 1$ then $a^{2g}+a \ge 0$ and therefore
$$\mathrm{selm}_{2g}(a)=(a^{2g}+a)+1 \ge 1>0.$$
If $a\in \BR, |a| < 1$  then $a+1>0$ and therefore
$$\mathrm{selm}_{2g}(a)=a^{2g}+(a+1) \ge a+1>0.$$

Let us assume that
 $g$ is {\sl not} congruent to $1 \bmod 3$.  Then $2g$ 
is {\sl not} congruent to $2 \bmod 3$ and therefore, by a theorem of Selmer \cite[Th. 1]{Sel56},
$\mathrm{selm}_{2g}(x)$ is irreducible over $\BQ$.  Notice that the coefficient of the trinomial 
$\mathrm{selm}_{2g}(x)$ at $x$ and the constant term are relatively prime, square free and coprime
to both $2g$ and $2g-1$.
It follows from \cite{NV} (see also \cite[Cor. 2 on p. 233]{Osada}) applied to $a_0=b_0=c=1,  n=2g$ that 
 $\mathrm{Gal}(\mathrm{selm}_{2g}(x))=\mathbf{S}_{2g}$. 
 
 It follows from Example \ref{Alt} 
that if a positive integer $g$ is {\sl not} congruent to $1 \bmod 3$ then
$$E=M_g:=K_{\mathrm{selm}_{2g}}=\BQ[x]/\mathrm{selm}_{2g}(x)\BQ[x]$$
is a degree $2g$ purely imaginary field that has no proper subfields except $\BQ$.

Now the construction of Subsection \ref{explicitexamples} gives us for all $g\ge 2$
that are {\sl not} congruent to  $1 \bmod 3$, 
a  $g$-dimensional  {\sl special} complex torus $T(\mathrm{selm}_{2g})$
with endomorphism algebra $M_g$.

Notice that if $g\ge 5$ is {\sl not} congruent to $1 \bmod 3$ then  $g$-dimensional special complex tori $T(\exp_{2g})$ and $T(\mathrm{selm}_{2g})$ are {\sl not}  isogenous.
Indeed, suppose that they are isogenous. Then their endomorphism algebras (which are actually number fields) $K_g$ and $M_g$ are isomorphic. 
It follows from \cite[the last assertion of Cor. 2 on p. 233]{Osada} (applied to $a_0=b_0=c=1,  n=2g$) that all the ramification indices in the field extension $M_g/\BQ$ do not exceed $2$. On the other hand,
 it is proven in \cite[Sect. 5]{Zar03} that there is a prime $p$ that enjoys the following properties.

\begin{itemize}
\item
$g+1\le p \le 2g+1$.
\item
One of ramification indices over $p$ in the field extension $K_g/\BQ$  is divisible by $p$.
In particular, this index 
$$\ge p\ge g+1\ge 5+1=6>2.$$
\end{itemize}
This implies that number fields
$K_g$ and $M_g$ are not isomorphic. The obtained contradiction proves that the tori  $T({\exp_{2g}})$ and $T(\mathrm{selm}_{2g})$ are {\sl not}  isogenous.

 \section{\bf Polynomials with doubly transitive Galois group}
\label{MoriZ}
The following construction was inspired by so called {\sl Mori polynomials} \cite{Mori, Zar16}.
As above, $g\ge 2$ is an integer, hence $2g-1\ge 3$. Let us fix

\begin{itemize}
\item
 a prime divisor $l$ of $2g-1;$  
 \item a prime $p$ that 
  is congruent to $1$ modulo  $2g-1;$ 
 \item
 an integer $b$ that is {\sl not} divisible by $l$  and that is a  primitive root $\bmod \  p;$
 \item an integer $c$ that is {\sl not} divisible by $l$. 
 \end{itemize}
 
 We call such a 
 $(l,p,b,c)$  a {\it $g$-admissible quadruple}.
 
 \begin{remark}
 Let $g \ge 2$ and $l$ be any prime divisor  of $2g-1$.
 In light of  Dirichlet's Theorem about primes in arithmetic progressions (which allows us to choose $p$) and Chinese Remainder Theorem (which allows us to choose $b$), there are infinitely many
 $g$-admissible quadruples  $(l,p,b,c)$.
 \end{remark}
 
Now let us consider a monic degree $2g$  polynomial
\begin{equation}
\label{mori3}
f_g(x)= f_{g,l,p,b,c}(x):=x^{2g}-bx-\frac{pc}{l^l}\in \BZ[1/l][x]\subset \BQ[x].
\end{equation}

\begin{Lemma}
\label{notRealM}
\begin{itemize}
\item[(i)]
The polynomial $f_g(x)=f_{g,l,p,b,c}(x)$ is irreducible over the field $\BQ_l$ of $l$-adic numbers and therefore over $\BQ$.
\item[(ii)]
The polynomial $\left(f_g(x)\bmod \ p\right) \in \mathbb{F}_p[x]$ is a product $x\left(x^{2g-1}-b\bmod p \right)$ of a linear factor
$x$ and an irreducible (over $\mathbb{F}_p$) degree $2g-1$ polynomial $x^{2g-1}-(b\bmod p)$.
\item[(iii)] Let $\mathrm{Gal}(f_g)$ be the Galois group of $f_g(x)$ over $\BQ$ viewed as a
transitive subgroup of  $\mathrm{Perm}(\RR_{f_g})$.  

Then  transitive $\mathrm{Gal}(f_g)$ contains a
permutation $\sigma$ that is a cycle of length $2g-1$. In particular, $\mathrm{Gal}(f_g)$ is a doubly
transitive permutation subgroup of $\mathrm{Perm}(\RR_{f_g})$. 
\item[(iv)]
The polynomial $f_g(x)$ has no real roots if and only if
\begin{equation}
\label{cNegative}
c<\frac{l^l\left(\frac{b}{2g}\right)^{1/(2g-1)}\left(\frac{b}{2g}-1\right)}{p}.
\end{equation}
\item[(v)]
Let $\ell$ be a prime that divides $b$, does not divide $2glp$, and such that $c$ is congruent to $\ell$ modulo $\ell^2$.
Then the discriminant of  the number field $\BQ[x]/f_g(x)\BQ[x]$ is divisible by $\ell$.
\item[(vi)]
Let $\ell$ be a prime that divides $c$ and does not divide $(2g-1)pb$.
Then the discriminant of the number field $\BQ[x]/f_g(x)\BQ[x]$ is not divisible by $\ell$.
\end{itemize}
\end{Lemma}

\begin{remark}
\label{decreaseC}
Let $(l,p,b, c)$ be a $g$-admissible quadruple.  Let $N$ be a positive integer such that
$$N>\frac{l^l\left(\frac{b}{2g}\right)^{1/(2g-1)}\left(\frac{b}{2g}-1\right)}{p}-c.$$
\begin{itemize}
\item[(1)]
Replacing $c$ by $c_1=c-Nlp$, we get a $g$-admissible quadruple $(l,p,b, c_1)$
such that the corresponding polynomial $f_{l,g,p,b,c_1}(x)$ has no real roots, in light of  Lemma \ref{notRealM}(iii)
and inequality \eqref{cNegative}. 
\item[(2)]
Let  $\ell$ be a prime that satisfies conditions (v) (respectively (vi)) of Lemma \ref{notRealM}  with respect to $(l,p,b, c)$. 
Let $c_2=c-Nlp\ell^2$. Then $(l,p,b, c_2)$ is also a  $g$-admissible quadruple and  $f_{l,g,p,b,c_2}(x)$ has no real roots, in light of the previous remark
(applied to $N\ell^2$ instead of $N$).
 In addition,  $c_2$ is congruent to $\ell$ modulo $\ell^2$
(respectively, is not divisible by $\ell$). In other words,
$\ell$ also satisfies the congruence properties  similar to conditions (v) (respectively to (vi)) of Lemma \ref{notRealM}  where
$c$ is replaced by $c_2$.
It follows from Lemma \ref{notRealM}(v) (respectively (vi)) that
 the discriminant of $\BQ[x]/f_{l,g,p,b,c_2}(x)\BQ[x]$ is divisible by $\ell$ (respectively, not divisible by $\ell$).
\end{itemize}
\end{remark}

\begin{proof}[Proof of Lemma \ref{notRealM}]

(i)  The $l$-adic Newton polygon of $f_g(x)$ consists of one
segment  with endpoints $(0,-l)$ and $(2g,0)$, which are its only integer points, since prime $l$ does {\sl not} divide $2g$.
Now the irreducibility of $f_g(x)$ follows from Eisenstein–Dumas Criterion (\cite[Corollary
3.6, p. 316]{Mott}, \cite[p. 502]{Gao}).

(ii) The conditions on $b$ and $p$ imply that for each divisor $d>1$ of $2g-1$ the residue $b  \bmod p$ is {\sl not} a $d$th power  $m^d$  for any $m\in\BF_p$.
It follows from theorem 9.1 of \cite[Ch. VI, Sect. 9]{Lang} that the polynomial $x^{2g-1}-(b \mod p)$ is {\sl irreducible} over $\BF_p$
and therefore its Galois group over $\BF_p$ is a cyclic group of order $2g-1$. 

(iii) Let us consider the reduction
$$\bar{f}_g(x) =\left( f_g(x)\bmod \ p \right)\in \BF_p[x]$$ of $f_g(x)$ modulo $p$. Clearly, $\bar{f}_g(x)=x(x^{2g-1}-\left(b \bmod \ p)\right)$ is a product in $\BF_p[x]$
of relatively prime linear $x$ and irreducible $x^{2g-1}-(b \bmod p)$. This implies that $\BQ\left(\RR_{f_g}\right)/\BQ$ is unramified at $p$ and a
corresponding Frobenius element in 
$$\mathrm{Gal}\left(\BQ(\RR_{f_g}\right)/\BQ)=\mathrm{Gal}(f_g)\subset \mathrm{Perm}\left(\RR_{f_g}\right)$$
 is a cycle of length $2g-1$.  This proves (iii).

(iv). Since $f_g(x)$ has even degree and positive leading coefficient, it reaches its smallest value on $\BR$ at a certain real point that is a zero of its
derivative $f_g^{\prime}(x)=2g x^{2g-1}-b$. The only real zero of $f_g^{\prime}(x)$ is $\beta=(b/2g)^{\frac{1}{2g-1}}$. Hence, $f_g(x)$ has no real roots if and only if
$f_g(\beta)>0$. We have
$$f_g(\beta)=\beta^{2g}-b\beta-\frac{pc}{l^l}=\left(\frac{b}{2g}\right)^{2g/2g-1}-b \left(\frac{b}{2g}\right)^{1/(2g-1)}-\frac{pc}{l^l}=$$
$$\left(\frac{b}{2g}\right)^{1/(2g-1)}\left(\frac{b}{2g}-1\right)-\frac{pc}{l^l}.$$
This implies that $f_g(\beta)>0$ if and only if
$$\left(\frac{b}{2g}\right)^{1/(2g-1)}\left(\frac{b}{2g}-1\right)>\frac{pc}{l^l},$$
 i.e.,
$$c<\frac{l^l}{p} \left(\frac{b}{2g}\right)^{1/(2g-1)}\left(\frac{b}{2g}-1\right).$$
This proves (iv).

(v)-(vi).   Let us consider the degree $2g$ number field  $E:=\BQ[x]/f_g(x)\BQ[x]$ and 
its  discriminant $\mathbf{\Delta}_E \in \BZ$.
The formula for the discriminant of a trinomial \cite[Example 834]{FS} tells us that the discriminant $\mathrm{Discr}(f_g)$ of $f_g(x)$ is
\begin{equation}
\label{discriminant}
\mathrm{Discr}(f_g)=(-1)^{g(2g-1)}(2g)^{2g} \left(\frac{pc}{l^l}\right)^{2g-1}+(-1)^{(2g-1)(g-1)}(2g-1)^{2g-1} b^{2g}=
\end{equation}
$$\pm \left(\frac{p}{l^l}\right)^{2g-1} (2g)^{2g} c ^{2g-1} \mp (2g-1)^{2g-1} b^{2g}\in \BZ[1/l].$$
Notice that there is a {\sl nonzero} rational number $r$ such that
\begin{equation}
\label{discField}
r^2 \cdot \mathbf{\Delta}_E=\mathrm{Discr}(f_g)
\end{equation}
 (see, e.g., \cite[Algebraic Extensions, Sect. 2.3, especially,  formula 2.12]{BS} applied to $k=\BQ$ and $K=E$).

 In the case of (v),  there are integers $c_1, b_1 \in \BZ$ such that 
$$c=\ell(1+c_1\ell), \  b=\ell b_1.$$ It follows from \eqref{discriminant} that
$$\mathrm{Discr}(f_g)=\ell^{2g-1} \cdot u_1+\ell^{2g} u_2$$
where $u_1\in \BZ[1/l]$ is an $\ell$-adic unit and $u_2\in \BZ$ is an integer. This implies that
$\mathrm{Discr}(f_g)=\ell^{2g-1} u$ where  $u \in \BQ$ is an $\ell$-adic unit.  Since $2g-1$ is odd, it follows from \eqref{discField} that $\mathbf{\Delta}_E$
 is divisible by $\ell$, which proves (v).

In the case of (vi),  it follows from  \eqref{discriminant} that 
$$\mathrm{Discr}(f_g)=\ell^{2g-1}v_1+v_2$$
where $v_1\in \BZ[1/l]$ is an $\ell$-adic unit and $v_2\in \BZ$ is an integer {\sl not} divisible by $\ell$.  This implies that     $\mathrm{Discr}(f_g)\in \BZ[1/l]$ is an $\ell$-adic unit.
  Taking into account that $\ell \ne l$, we obtain that the reduction modulo $\ell$
  $$f_g(x) \bmod \ \ell \in (\BZ[1/l]/\ell\BZ[1/l])[x]=\BF_{\ell}[x]$$
of  $f_g(x)$ is a degree $2g$ monic polynomial with coefficients in $\BF_{\ell}$ and without repeated  roots. It follows from \cite[Ch. III, Sect 2, Th. 23 on p. 129]{FT}
  (applied to $\mathfrak{o}=\BZ[1/l]$ and  $\mathfrak{p}=\ell\BZ[1/l]$) that
  the prime ideal $\ell\BZ[1/l]$ of  the Dedekind ring $\BZ[1/l]$ is unramified in $E$. This means that the {\sl discriminant ideal} $\mathbf{\Delta}_E\cdot \BZ[1/l]$ of $\BZ[1/l]$ is {\sl not} contained in 
  $\ell\BZ[1/l]$. It follows that $\mathbf{\Delta}_E$
   is {\sl not} divisible by $\ell$, which 
 proves (vi).
\end{proof}

Now assume that we have chosen $c$ in such a way that inequality \eqref{cNegative} holds. It can be done, in light of Remark \ref{decreaseC}. Then we have: 
\begin{itemize}\item $f_g(x)$ is irreducible over $\BQ$ and has no real roots
(\lemref{notRealM}(i));
\item the group $\mathrm{Gal}(f_g)$ is  doubly
transitive (\lemref{notRealM}(iii)); 
\item  the group $\mathrm{Gal}(f_g)$ is primitive (Remark \ref{permuatation}).
\end{itemize}

 It follows from Lemma \ref{notRealM}
that 
$$E=L_g=L_{g,l,p,b,c}:=\BQ[x]/f_{g,l,p,b,c}(x)\BQ[x]$$
is a degree $2g$ purely imaginary field that has no proper subfields except $\BQ$.

Now the construction of Section \ref{explicitexamples} gives us for all $g\ge 2$
a {\sl special}  $g$-dimensional 
complex torus
$T_{g,l,p,b,c}:=T(f_{g,l,p,b,c})$
with endomorphism algebra $L_{g,l,p,b,c}$.

\begin{remark}
$g$-admissible quadruples $(l,p,b,c)$ and polynomials of odd degree $2g+1$ similar to $f_{g,l,p,b,c}(x)$
were introduced by S. Mori \cite{Mori} for $l=2$, who proved in this case analogs of the assertions (i)-(iii) of Lemma \ref{notRealM}.
 (See also \cite{Zar16}).

\end{remark}

\section{Isogeny classes}
\label{Infinite}

 Let $g \ge 2$ be an integer. The aim of this section is to   construct infinitely many mutually non-isogenous special $g$-dimensional complex tori.

Let us choose a $g$-admissible quaduple $(l,p,b,c)$ that satisfies \eqref{cNegative}.   The construction of  Section \ref{MoriZ} gives us a special complex torus
$T^{(1)}:=T_{g,l,p,b,c}$ of dimension $g$.  Suppose that $n$ is a positive integer and we have already constructed $n$ mutually non-isogenous $g$-dimensional special complex tori
$$T^{(k)}=T_{g,l,p,b_k,c_k},  \ 1 \le k \le n$$
where each $(l,p,b_k,c_k)$ is a  $g$-admissible quaduple such that $f_{g,l,p,b_k,c_k}(x)$ has no real roots.  In particular, the endomorphism algebra of $T^{(k)}$ is
 the  purely imaginary number field $L_{g,l,p,b_k,c_k}$.

Let us choose 
\begin{itemize}
\item  an odd  prime $\ell \ne l,p$ that does {\sl not} divide $g$, and is {\sl unramified}
in all number fields $L_{g,l,p,b_k,c_k}$ ($1\le k \le n$), i.e.,  does not divide the discriminant of any   $L_{g,l,p,b_k,c_k};$
\item 
an integer $b_{n+1}$  that is {\sl not} divisible by $l$ and   is a  primitive root $\bmod \ p$.
\end{itemize}
 Assume additionally,
that $b_{n+1}$ is divisible by $\ell$. Since all three primes $l,p,\ell$ are distinct, such a $b_{n+1}$ does exist, thanks to Chinese Remainder Theorem.
Now let us choose an integer $c_{n+1}$ that is  not divisible by $l$ and congruent to $\ell$ modulo $\ell^2$. Then 
 $(l,p,b_{n+1},c_{n+1})$ is a $g$-admissible quadruple such that the discriminant of the number field $L_{g,l,p,b_{n+1},c_{n+1}}$
 is {\sl divisible} by $\ell$, thanks to Lemma \ref{notRealM}(v). According to Remark \ref{decreaseC}, one may also choose $c_{n+1}$ in such a way
 that  $f_{g,l,p,b_{n+1},c_{n+1}}(x)$ has no real roots., i.e., the field  $L_{g,l,p,b_{n+1},c_{n+1}}$ is purely imaginary.
 This gives us 
a  special $g$-dimensional complex torus
$T^{(n+1)}=T_{g,l,p,b_{n+1},c_{n+1}}$, whose endomorphism algebra $\End^0(T^{(n+1)})$ is  isomorphic to the field
$L_{g,l,p,b_{n+1},c_{n+1}}$, which is {\sl ramified} at $\ell$.

 Our choice of $\ell$
implies that $L_{g,l,p,b_{n+1},c_{n+1}}$ is not isomorphic to any of $L_{g,l,p,b_k,c_k}$ with $k \le n$. It follows that $T^{(n+1)}$ is {\sl not} isogenous to any
of $T^{(k)}$ with $k \le n$.
In light of results of Section \ref{MoriZ}, all $T^{(1)}, \dots, T^{(n)},  T^{(n+1)} \dots $ are {\sl special} $g$-dimensional {\sl mutually non-isogenous}
complex tori.

\end{document}